\documentclass[11pt]{article}

\usepackage{geometry}
\usepackage{amsthm, amsmath, amsfonts, amssymb, mathtools, mathrsfs}
\usepackage{url, hyperref}
\usepackage{xcolor}
\usepackage{stmaryrd}   
\usepackage{bbm}        

\newtheorem{theorem}{Theorem}[section]
\newtheorem*{theorem*}{Theorem}
\newtheorem{lemma}[theorem]{Lemma}

\newtheorem{corollary}[theorem]{Corollary}

\theoremstyle{remark}
\newtheorem{definition}[theorem]{\bf Definition}
\newtheorem{remark}[theorem]{\bf Remark}

\newtheorem{question}[theorem]{\bf Question}

\makeatletter
\newcommand{\subjclass}[2][2020]{%
	\let\@oldtitle\@title%
	\gdef\@title{\@oldtitle\footnotetext{#1 \emph{Mathematics subject classification}: #2}}%
}
\newcommand{\keywords}[1]{%
	\let\@@oldtitle\@title%
	\gdef\@title{\@@oldtitle\footnotetext{\emph{Key words and phrases}: #1}}%
}
\makeatother
\renewcommand{\ge}{\geqslant}
\renewcommand{\le}{\leqslant}

\DeclareMathOperator{\LiA}{Li_A}
\DeclareMathOperator{\wt}{wt}
\DeclareMathOperator{\dep}{dep}
\DeclareMathOperator{\trdeg}{trdeg}

\newcommand{\bbF}{\mathbb{F}}
\newcommand{\bbZ}{\mathbb{Z}}

\newcommand{\cZ}{\mathcal{Z}}
\newcommand{\cI}{\mathcal{I}}

\newcommand{\fs}{\mathfrak{s}}
\newcommand{\fh}{\mathfrak{h}}

\newcommand{\ft}{\mathfrak{t}}

\newcommand{\sS}{\mathscr{S}}

\newcommand{\bS}{\boldsymbol{S}}
\newcommand{\brho}{\boldsymbol{\rho}}

\newcommand{\sfE}{\mathsf{E}}
\newcommand{\sfC}{\mathsf{C}}

\numberwithin{equation}{section}

\allowdisplaybreaks

\title{A polynomial basis for the multizeta algebra in positive characteristic}
\author{Li Lai}
\date{}

\begin{document}
\maketitle

\begin{abstract}
Let $K=\mathbb{F}_q(\theta)$, and let $\mathcal{Z}$ be the $K$-algebra generated by Thakur's multiple zeta values $\zeta_A(\fs)$.
We prove that $\mathcal{Z}$ is a polynomial algebra and construct an explicit polynomial basis of $\mathcal{Z}$ over $K$.
We also determine the transcendence degree of $K[\zeta_A(\mathfrak{s}): \wt(\mathfrak{s}) \leqslant w]$ over $K$ for every integer $w \ge 1$, generalizing a result of Ngo Dac--Nguyen Chu--Pham.

The proof uses Chang's grading theorem, the linear basis theorem proved independently by Chang--Chen--Mishiba and Im--Kim--Le--Ngo Dac--Pham, and the carry relations of Im--Kim--Ngo Dac. 
The key step is to show that suitable derivations obtained from deconcatenation and linear functionals descend through the carry relations to $\mathcal{Z}$.
\end{abstract}

\section{Introduction}
\subsection{Background}

The arithmetic nature of special values of the Riemann zeta function
\[
\zeta(s) \coloneq \sum_{n=1}^{\infty} \frac{1}{n^s} \quad (s=2,3,\ldots)
\]
has attracted considerable attention.
Euler proved that $\zeta(s)/(2\pi\sqrt{-1})^s \in \mathbb{Q}^{\times}$ for any positive even integer $s$. 
A folklore conjecture predicts that $\pi,\zeta(3),\zeta(5),\zeta(7),\ldots$ are algebraically independent over $\mathbb{Q}$, but little is known.
See the classical results \cite{Apery1979,BR2001,Zudilin2001} and some recent developments \cite{Fischler2026,Lai2026}, to mention just a few.

In 1935, Carlitz \cite{Carlitz1935} introduced analogues of Riemann zeta values in positive characteristic.
Let $A=\bbF_q[\theta]$ be the polynomial ring in one indeterminate $\theta$ over a finite field $\bbF_q$ of characteristic $p>0$.
Denote by $A_+$ the set of monic polynomials in $A$.
Let $K=\bbF_q(\theta)$ be the fraction field of $A$ and $K_{\infty}=\bbF_q(\!(1/\theta)\!)$ the completion of $K$ at $\infty$.
The Carlitz zeta values are defined for $s=1,2,\ldots$ by
\[
\zeta_A(s) \coloneq \sum_{a \in A_+} \frac{1}{a^s} \in K_{\infty}. 
\]
Carlitz proved that $\zeta_A(s)/\widetilde{\pi}^s \in K^{\times}$ for any positive integer $s$ such that $q-1 \mid s$, where $\widetilde{\pi} \in \overline{K_{\infty}}^{\textup{alg}}$ is an analogue of $2\pi\sqrt{-1} \in \mathbb{C}$.
A breakthrough was achieved by Yu \cite{Yu1991} in 1991; he showed that $\zeta_A(s)$ is transcendental over $K$ for any positive integer $s$, and $\zeta_A(s)/\widetilde{\pi}^s$ is transcendental over $K$ for any positive integer $s$ such that $q-1 \nmid s$.

In 2007, Chang and Yu \cite{CY2007} determined all algebraic relations among Carlitz zeta values. 
They proved the following theorem.
Let $\bbZ_+$ be the set of positive integers.

\begin{theorem}[Chang--Yu \cite{CY2007}]\label{thm_CY} 
The set
\[
\{ \widetilde{\pi} \} \bigcup \Big\{ \zeta_A(s) : s \in \bbZ_+, p \nmid s, q-1 \nmid s  \Big\}
\]
is algebraically independent over $K$.
Moreover, for any $w \ge 1$, 
\[
K\left[ \widetilde{\pi},\zeta_A(1),\zeta_A(2),\ldots,\zeta_A(w)\right] = K\left[ \widetilde{\pi}, \zeta_A(s) : 1 \le s \le w, p \nmid s, q-1 \nmid s\right].
\]
In particular, the transcendence degree of (the fraction field of) $K\left[ \widetilde{\pi},\zeta_A(1),\zeta_A(2),\ldots,\zeta_A(w)\right]$ over $K$ is 
\[
w - \left\lfloor \frac{w}{p} \right\rfloor - \left\lfloor \frac{w}{q-1} \right\rfloor + \left\lfloor \frac{w}{p(q-1)} \right\rfloor + 1.
\] 
\end{theorem}
Thus, we have a satisfactory understanding of the algebraic relations among Carlitz zeta values.

Since the 1990s, multiple zeta values (MZVs) have attracted increasing interest.
For positive integers $s_1,\ldots,s_k$ with $s_1 > 1$, the multiple zeta value $\zeta(s_1,\ldots,s_k)$ is defined by
\[
\zeta(s_1,\ldots,s_k) \coloneq \sum_{n_1 > n_2 > \cdots > n_k \ge 1} \frac{1}{n_1^{s_1}n_2^{s_2}\cdots n_k^{s_k}} \in \mathbb{R}.
\]
We refer the reader to the monograph \cite{BF} for a systematic introduction to MZVs.

In 2004, Thakur \cite{Thakur2004} introduced analogues of MZVs in positive characteristic. 
Let $\cI$ denote the disjoint union $\{\emptyset\} \cup \bigcup_{k=1}^{\infty} \bbZ_+^k$. 
Elements of $\cI$ are called tuples.
For any tuple $\fs =(s_1,\ldots,s_k) \in \cI$, the \emph{depth} $\dep(\fs)$ and \emph{weight} $\wt(\fs)$ are defined by $\dep(\fs) = k$ and $\wt(\fs)=\sum_{i=1}^{k} s_i$, with the conventions $\dep(\emptyset)=0$ and $\wt(\emptyset)=0$.
Let $\cI_{>0} = \cI \setminus \{\emptyset\}$ be the set of nonempty tuples.
For any $\fs=(s_1,\ldots,s_k) \in \cI_{>0}$, Thakur's MZV is defined by
\[
\zeta_A(\fs) \coloneq \sum_{\substack{a_1, a_2, \ldots, a_k \in A_+ \\  \deg a_1 > \deg a_2 > \cdots > \deg a_k}} \frac{1}{a_1^{s_1}a_2^{s_2}\cdots a_k^{s_k}} \in K_{\infty}.
\]
By convention, $\zeta_A(\emptyset) \coloneq 1$.
For any $w \ge 0$, let $\cI_w$ denote the set of tuples of weight $w$.
Define the $K$-linear spaces
\[
\cZ \coloneq \operatorname{Span}_K\left\{ \zeta_A(\fs) : \fs \in \cI \right\} \subset K_{\infty}, \quad\text{and}\quad \cZ_{w} \coloneq \operatorname{Span}_K\left\{ \zeta_A(\fs) : \fs \in \cI_w \right\} \subset K_{\infty}, w \ge 0.
\]
Thakur proved that $\zeta_A(\fs_1)\cdot\zeta_A(\fs_2) \in \cZ_{\wt(\fs_1)+\wt(\fs_2)}$ for any $\fs_1,\fs_2 \in \cI$.
After Chen's work \cite{Chen2015}, these product relations were made explicit in the form $\zeta_A(\fs_1)\cdot\zeta_A(\fs_2) = \zeta_A(\fs_1 *^{\zeta} \fs_2)$; see \cite{Thakur2017}. 
We call them $q$-shuffle relations.
Therefore, $(\cZ,+,\cdot)$ is a $K$-algebra.

A basic question is to determine all algebraic relations among Thakur's MZVs. 
A major advance was made by Chang \cite{Chang2014} in 2014:
\begin{theorem}[Chang \cite{Chang2014}]\label{thm_Chang}
$\cZ$ is a graded algebra over $K$; that is, $\cZ = \bigoplus_{w=0}^{\infty} \cZ_w$.
\end{theorem}
Chang's theorem shows that every algebraic relation among Thakur's MZVs follows from the $q$-shuffle relations and $K$-linear relations among values of the same weight.
Todd \cite{Todd2018} and Thakur \cite{Thakur2017} formulated a $K$-linear basis conjecture for the spaces $\cZ_w$.
In 2021, Ngo Dac \cite{Ngo2021} proved the spanning part of this
conjecture, and made important progress toward proving linear independence.
It was subsequently proved in full by Im--Kim--Le--Ngo Dac--Pham \cite{IKLNP2024} and, independently, by Chang--Chen--Mishiba \cite{CCM2023}.
Define  
\[
[q]^{\bullet} \coloneq \{\emptyset\} \bigcup \Big\{ (s_1,\ldots,s_k) \in \cI_{>0} : k \in \bbZ_+, 1 \le s_1,\ldots,s_k \le q \Big\}.
\]
\begin{theorem}[Chang--Chen--Mishiba \cite{CCM2023}, independently Im--Kim--Le--Ngo Dac--Pham \cite{IKLNP2024}]\label{thm_CCMIKLNP}
For any weight $w \ge 0$, we have $\dim_{K} \cZ_w = d_w$, where the sequence $\{d_w\}_{w \ge 0}$ is defined by the generating function 
\begin{equation}\label{def_dn}
\sum_{w=0}^{\infty} d_w X^w = \frac{1-X^{q}}{1-X-X^2-\cdots-X^{q}}.
\end{equation}
Moreover, $\left\{ \zeta_A(\fs) : \fs \in \mathcal{B}_w^{\textup{T}} \right\}$ is a $K$-basis of $\cZ_w$, where $\mathcal{B}_0^{\textup{T}} = \{\emptyset\}$ and
\begin{equation}\label{def_TT}
\mathcal{B}_w^{\textup{T}} = \left\{ (\fh,z) \in \cI_w : \fh \in [q]^{\bullet}, 1 \le z \le q-1 \right\}, \quad w \ge 1.
\end{equation}
\end{theorem}

These results mentioned above lead to a natural question: can we generalize Theorem \ref{thm_CY} to the multizeta case? 
The purpose of this paper is to address this question.
See \cite{Mishiba2015,Mishiba2017,NNP2026} for some known results related to this question.
We are also interested in the algebraic structure of $\cZ$; see \cite{IKLNP2023+,NNP2026}.

The closely related Carlitz multiple polylogarithm values (CMPLVs) are useful in the study of Thakur's MZVs.
Following Chang \cite{Chang2014}, we define for each $\fs \in \cI_{>0}$ the CMPLV 
\[
\LiA(\fs) \coloneq \sum_{n_1>n_2>\cdots>n_k \ge 0}  \frac{ 1 }{\ell_{n_1}^{s_1} \ell_{n_2}^{s_2} \cdots \ell_{n_k}^{s_k}} \in K_{\infty},
\]
where $\ell_n = \prod_{i=1}^{n} (\theta-\theta^{q^i})$ for $n \in \bbZ_+$ and $\ell_0=1$.
We adopt the convention $\LiA(\emptyset) \coloneq 1$.

\begin{remark}\label{rmk1.3}
It is well known that $\LiA(\fs)=\zeta_A(\fs)$ for any $\fs \in [q]^{\bullet}$; see Thakur \cite{Thakur2009}.
In particular, $\LiA(\fs)=\zeta_A(\fs)$ for any $\fs \in \mathcal{B}_w^{\textup{T}}$.
Thus, Theorem \ref{thm_CCMIKLNP} implies $\{ \LiA(\fs) : \fs \in \mathcal{B}_w^{\textup{T}} \}$ is a $K$-basis of $\cZ_w$ for any $w \ge 0$.
\end{remark}

The following theorem of Im--Kim--Ngo Dac \cite{IKN2026+} connects Thakur's MZVs with CMPLVs.

\begin{theorem}[{Im--Kim--Ngo Dac \cite{IKN2026+}}]\label{thm_FpSpan}
For any weight $w \ge 0$,
\begin{equation}\label{eqn_FpSpan}
\operatorname{Span}_{\bbF_p}\left\{ \zeta_{A}(\fs) : \fs \in \cI_w  \right\} = \operatorname{Span}_{\bbF_p}\left\{ \LiA(\fs) : \fs \in \cI_w  \right\}.
\end{equation}
As a corollary,
\[
\cZ_w = \operatorname{Span}_{K}\left\{ \LiA(\fs) : \fs \in \cI_w  \right\}.
\]
\end{theorem}

\begin{proof}
In \cite[\S 3]{IKN2026+}, Im, Kim, and Ngo Dac proved that
\[
\operatorname{Span}_{\bbF_q}\left\{ \zeta_{A}(\fs) : \fs \in \cI_w  \right\} = \operatorname{Span}_{\bbF_q}\left\{ \LiA(\fs) : \fs \in \cI_w  \right\}.
\]
However, their proof actually leads to the stronger conclusion \eqref{eqn_FpSpan}.
\end{proof}

\subsection{Main results}
To state our main results, we first introduce some notation.
Let $a_n$ (for $n \ge 1$) be the integers uniquely determined by the following identity in the formal power series ring $\mathbb{Z}\llbracket X \rrbracket$:
\begin{equation}\label{def_an}
\frac{1-X^{q}}{1-X-X^2-\cdots-X^q} = \prod_{n=1}^{\infty} \frac{1}{(1-X^n)^{a_n}}.
\end{equation}
These integers $a_n$ are nonnegative; see Lemma \ref{lem4.3}.
For $n \ge 1$ and $1 \le j \le a_n$, define 
\begin{equation}\label{def_znj}
z_{n,j} \coloneq \sum_{i=1}^{d_n} D_1^{i\cdot4^{nj}} \zeta_A(\ft_{n,i}) \in \cZ_{n},
\end{equation}
where $D_1=\theta^q-\theta \in K$, and 
\[
\ft_{n,1} \prec \ft_{n,2} \prec \cdots \prec \ft_{n,d_n} 
\]
are the Todd--Thakur tuples of weight $n$ (see Equation \eqref{def_TT}) listed in ascending order with respect to the (depth, lex)-order (see Definition \ref{def_deplex}). 
When $a_n=0$, there is no $z_{n,j}$.

Our main result is the following.

\begin{theorem}\label{thm_main}
The elements $z_{n,j}$ $(n \in \bbZ_+, 1 \le j \le a_n)$ freely generate the algebra $\cZ$ over $K$.
Moreover, for any weight $w \ge 1$, the elements $z_{n,j}$ $(1 \le n \le w, 1 \le j \le a_n)$ freely generate the algebra $\cZ_{\le w}^{\textup{alg}} \coloneq K[\zeta_A(\mathfrak{s}): \wt(\mathfrak{s}) \leqslant w]$.
In particular, both $\cZ$ and $\cZ_{\le w}^{\textup{alg}}$ are polynomial algebras, and
\[
\trdeg_{K} \cZ_{\le w}^{\textup{alg}} = \sum_{n=1}^{w} a_n.
\]
\end{theorem}

\begin{remark}
Taking the logarithmic derivative of \eqref{def_an} and using the M\"obius inversion formula, we obtain
\[
a_n = \frac{1}{n} \sum_{d \mid n} \mu\left(\frac{n}{d}\right) \left( -q \cdot \mathbbm{1}_{q \mid d} + \sum_{i=1}^{q} \lambda_i^{d} \right)
\]
for any $n \ge 1$, where $\mu(\cdot)$ is the M\"obius function, and $\lambda_1,\lambda_2,\ldots,\lambda_q \in \mathbb{C}$ are roots of the polynomial $X^q - X^{q-1} - X^{q-2} - \cdots - 1$ (they are all simple roots). 
Applying Rouch\'e's theorem to the polynomial $X^{q+1} -2X^q +1$, we deduce that exactly one of the roots $\lambda_1,\lambda_2,\ldots,\lambda_q \in \mathbb{C}$, denoted by $\lambda$, lies in the real interval $1 < \lambda < 2$; any other root has modulus at most $1$. 
Therefore,
\[
\trdeg_{K} \cZ_{\le w}^{\textup{alg}} \sim \frac{\lambda^{w+1}}{(\lambda-1)w} \quad\text{as $w \to \infty$}.
\]
\end{remark}

Note that Theorem \ref{thm_main} provides an explicit polynomial basis of $\cZ$, but this basis consists of linear combinations of Thakur's MZVs.
One may ask the following question:

\begin{question}\label{question}
Can we find an explicit polynomial basis of $\cZ$ consisting of individual MZVs of Thakur, with a simple combinatorial description? 
\end{question}

Ngo Dac--Nguyen Chu--Pham \cite{NNP2026} obtained partial results toward answering this question.
Their construction uses Lyndon words.
Question \ref{question} deserves further investigation.

We briefly outline the proof of Theorem \ref{thm_main} as follows.
Our starting point is the carry relations of Im--Kim--Ngo Dac \cite{IKN2026+}.
We recall them in Section \ref{sec_2} and establish some properties of them.
Section \ref{sec_3} is crucial.
We introduce a family of $K$-derivations on $\cZ$, which are induced by $K$-linear functionals on $\cZ$ satisfying certain properties that respect the carry relations.
In Section \ref{sec_4}, we prove in Theorem \ref{thm4.2} that homogeneous lifts of the elements in any homogeneous $K$-basis of 
\[
\cZ_{>0}/\cZ_{>0}\cZ_{>0}
\]
freely generate the $K$-algebra $\cZ$, where $\cZ_{>0} = \bigoplus_{w=1}^{\infty} \cZ_w$. 
The $K$-derivations introduced in Section \ref{sec_3} are used to prove the algebraic independence.
Finally, Section \ref{sec_5} proves Theorem \ref{thm_main} using Theorem \ref{thm4.2} and an elementary determinant argument.

\medskip

\noindent\textbf{Statement on AI use}: 
The initial proof was completely generated by ChatGPT (GPT-6 Astra, Open AI). 
The author spent about two weeks verifying, studying, and rewriting it.
Therefore, the author takes full responsibility but no credit for this work.

\section{Some properties of carry relations}\label{sec_2}

In this section, we first recall the carry relations due to Im--Kim--Ngo Dac \cite{IKN2026+};
they are $K$-linear relations among CMPLVs.
Then, we prove several useful properties of carry relations.
It is convenient to work with them as formal sums.

\subsection{Notation}
For any weight $w \ge 0$, define 
\[
\sS_w = \left\{ \bS=\sum_{\fs \in \cI_w} c_{\fs} \cdot \fs : c_{\fs} \in K \right\}
\]
to be the $K$-linear space of formal sums of tuples of weight $w$.
Two formal sums $\bS=\sum_{\fs \in \cI_w} c_{\fs} \cdot \fs$ and $\bS^\prime =\sum_{\fs \in \cI_w} c_{\fs}^\prime \cdot \fs$ are equal if and only if $c_{\fs}=c_{\fs}^\prime$ for every $\fs \in \cI_w$.

Define $\sS=\bigoplus_{w=0}^{\infty} \mathscr{S}_w$.
The stuffle product $*$ is a $K$-bilinear map $\sS \times \sS \longrightarrow \sS$ such that
\begin{align*}
& \emptyset * \bS = \bS * \emptyset = \bS, \\
& \fs * \ft = \left(s_1, \fs_{-} * \ft\right) + \left(t_1, \fs * \ft_{-} \right) + \left(s_1+t_1, \fs_{-} * \ft_{-} \right)
\end{align*}
for any $\bS \in \sS$ and any $\fs = \left(s_1, \fs_{-} \right)$, $\ft = \left(t_1,\ft_{-} \right) \in \cI_{>0}$, where $s_1,t_1 \in \bbZ_+$ and $\fs_{-},\ft_{-} \in \cI$.
Equipped with the stuffle product, $\sS$ is a graded $K$-algebra.

For any formal sum $\bS=\sum_{\fs \in \cI} c_{\fs} \cdot \fs \in \sS$ (where $c_{\fs} = 0$ for all but finitely many tuples $\fs \in \cI$), define the $\LiA$-evaluation 
\[
\LiA(\bS) = \LiA\left(\sum_{\fs \in \cI} c_{\fs} \cdot \fs\right) \coloneq \sum_{\fs \in \cI} c_{\fs}\LiA(\fs) \in K_{\infty}.
\]
By Theorem \ref{thm_FpSpan}, we actually have $\LiA(\bS) \in \cZ$.
It is well known that $\LiA(\fs_1 * \fs_2)=\LiA(\fs_1)\LiA(\fs_2)$ for any $\fs_1,\fs_2 \in \cI$.
Therefore, $\LiA: (\sS,+,*) \rightarrow (\cZ,+,\cdot)$ is a $K$-algebra homomorphism.

Every nonempty tuple $\fs = (s_1,\ldots,s_k) \in \cI_{>0}$ has the unique decompositions
\[
\fs = (\fs_{+},s_k) = (s_1,\fs_{-}),
\]
where $\fs_{+},\fs_{-} \in \cI$ and $s_1,s_k \in \bbZ_+$. 
For any positive integer $a$ and nonempty tuple $\fs \in \cI_{>0}$, define 
\[
\fs^{+a} \coloneq (\fs_{+},s_k+a) \quad\text{and}\quad {}^{a+}\fs \coloneq (a+s_1,\fs_{-}).
\]
Define $\emptyset^{+a}={}^{a+}\emptyset \coloneq (a)$.
By convention, $\fs^{+}=\fs^{+1}$ and ${}^{+}\fs = {}^{1+}\fs$.

Throughout this paper, $D_1 = \theta^q -\theta \in K$.

\subsection{Carry relations}

\begin{lemma}[carry relations of Im--Kim--Ngo Dac \cite{IKN2026+}]\label{lem_gIKN}
For any $\fh,\ft \in \cI$, define the following formal sum of tuples:
\begin{align}
\brho\left( \fh; \ft \right) &\coloneq \left( \fh, {}^{q+}\ft \right) + \mathbbm{1}_{\ft \neq \emptyset} \cdot \left( \fh, q, \ft \right) \notag\\
&\quad+D_1 \cdot \mathbbm{1}_{\fh \neq \emptyset} \cdot \left( \fh^{+}, (q-1)*\ft \right) + D_1 \cdot \left( \fh, 1, (q-1)*\ft \right) \in \sS_w, \label{def_rho}
\end{align}
where $w=\wt(\fh)+\wt(\ft)+q$.
Then
\[
\LiA\left( \boldsymbol{\rho}\left( \fh; \ft \right) \right) = 0.
\]
\end{lemma}

\begin{proof}
If we restrict to $\fh \in [q]^{\bullet}$, then this lemma is exactly \cite[Proposition 4.3]{IKN2026+} by Im, Kim, and Ngo Dac.
However, the same proof applies to every $\fh \in \cI$.
\end{proof}

For any $\fh,\ft \in \cI$, we call $\brho(\fh;\ft)$ a carry relation; see the explanation of the name in \cite[\S 2.1 and \S 3.1]{HHLL2026+}.
As already noted by Im--Kim--Ngo Dac in \cite{IKN2026+}, the carry relations behave well with respect to the (depth, lex)-order on $\cI$.
\begin{definition}\label{def_deplex}
The (depth, lex)-order $\prec$ on $\cI$ is defined as follows.
For any two tuples $\fs_1, \fs_2 \in \cI$, define $\fs_1 \prec \fs_2$ if and only if
\begin{itemize}
\item either $\operatorname{dep}(\fs_1) < \operatorname{dep}(\fs_2)$,
\item or $\operatorname{dep}(\fs_1) = \operatorname{dep}(\fs_2)$ and $\fs_1$ is lexicographically smaller than $\fs_2$.
\end{itemize}
Note that $\prec$ is a well-order on $\cI$. 
For any formal sum of tuples $\boldsymbol{S}=\sum_{\fs\in\cI} c_{\fs} \cdot \fs \in \mathscr{S}$ with $\boldsymbol{S} \neq \boldsymbol{0}$, the minimal tuple appearing in $\boldsymbol{S}$ is defined to be $\min_{\prec} \left\{ \fs \mid c_{\fs} \neq 0 \right\}$.

\end{definition}  

A carry relation $\brho(\fh;\ft)$ is called special if $\fh \in [q]^{\bullet}$.
The following theorem is implicitly proved in \cite{IKN2026+}.
For the reader's convenience, we include a short proof.

\begin{theorem}[Im--Kim--Ngo Dac]\label{lem_kerLi}
Special carry relations form a $K$-basis of $\ker(\LiA: \sS \rightarrow \cZ)$.
\end{theorem}

\begin{proof}
By Theorem \ref{thm_Chang}, we have
\[
\ker(\LiA: \sS \rightarrow \cZ) = \bigoplus_{w=0}^{\infty} \ker(\LiA: \sS_w \rightarrow \cZ_w).
\]
It remains to prove that 
\[
\mathcal{R}_w^{\textup{gIKN}} \coloneq \left\{ \brho(\fh;\ft) : \fh \in [q]^{\bullet}, \ft \in \cI, \wt(\fh) + \wt(\ft) + q = w \right\}
\]
is a $K$-basis of $\ker\left( \LiA: \sS_w \rightarrow \cZ_w \right)$ for every weight $w$. 

Note that the minimal tuple appearing in $\brho(\fh;\ft)$ is $\left( \fh, {}^{q+}\ft \right)$; see \cite[Lemma 4.5]{IKN2026+}.
For different pairs $(\fh;\ft) \in [q]^{\bullet} \times \cI$, the corresponding $\left( \fh, {}^{q+}\ft \right)$ are distinct.
Therefore, by \cite[Lemma 5.1]{HHLL2026+},  $\mathcal{R}_w^{\textup{gIKN}}$ is a $K$-linearly independent set.

By a simple generating-function argument, Theorem \ref{thm_CCMIKLNP}, and the rank-nullity theorem, we have
\begin{align*}
\sum_{w=0}^{\infty} \left(\#\mathcal{R}_w^{\textup{gIKN}}\right) X^w &= \frac{1}{1-X-X^2-\cdots-X^q} \cdot \frac{1-X}{1-2X} \cdot X^q = \frac{1-X}{1-2X} - \frac{1-X^q}{1-X-X^2-\cdots-X^q} \\
&=\sum_{w=0}^{\infty} \left(\dim_{K}\sS_w - \dim_K \cZ_w \right) X^w = \sum_{w=0}^{\infty} \left( \dim_K \ker\left( \LiA: \sS_w \rightarrow \cZ_w \right) \right) X^w.
\end{align*}
Thus, $\#\mathcal{R}_w^{\textup{gIKN}} = \dim_K \ker\left( \LiA: \sS_w \rightarrow \cZ_w \right)$, which completes the proof.
\end{proof}

The following simple lemma and its corollary will be used in Section \ref{sec_5}.
Recall that $\mathcal{B}_w^{\textup{T}}$ is the set of weight-$w$ Todd--Thakur tuples defined in \eqref{def_TT}.
Let $\bbF_p[D_1]_{< n}$ (respectively, $\bbF_p[D_1]_{\le n}$) denote the set $\{ \sum_{i=0}^{n-1} a_i D_1^{i} : a_i \in \bbF_p \}$ (respectively, $\{ \sum_{i=0}^{n} a_i D_1^{i} : a_i \in \bbF_p \}$).

\begin{lemma}\label{lem2.3}
For any $\fs \in \cI_{w}$ with $w \ge 1$, there exist $c_{\ft} \in \bbF_p[D_1]_{< 2^{w}}$ such that
\[
\LiA(\fs) = \sum_{\ft \in \mathcal{B}_w^{\textup{T}}} c_{\ft} \LiA(\ft).
\]
\end{lemma}

\begin{proof}
List the $2^{w-1}$ tuples in $\cI_w$ in descending order with respect to the (depth, lex)-order:
\[
\fs_1 \succ \fs_2 \succ \cdots \succ \fs_{2^{w-1}}.
\]
We prove the following stronger result by induction on $i$: for any $1 \le i \le 2^{w-1}$, 
\begin{equation}\label{1.2}
\LiA(\fs_i) \in \sum_{\ft \in \mathcal{B}_w^{\textup{T}}} \bbF_p[D_1]_{< i} \LiA(\ft).
\end{equation}
We have $\fs_1=(1,1,\ldots,1) \in \mathcal{B}_w^{\textup{T}}$, so \eqref{1.2} holds for $i=1$.

For $i>1$, if $\fs_i \in \mathcal{B}_w^{\textup{T}}$, then \eqref{1.2} holds trivially. 
If $\fs_i \notin \mathcal{B}_w^{\textup{T}}$, then there exist $\fh \in [q]^{\bullet}$ and $\ft \in \cI$ such that
\[
\fs_i = (\fh,{}^{q+}\ft), \quad \wt(\fh)+\wt(\ft)+q=w.
\]
By Lemma \ref{lem_gIKN}, we have
\begin{equation}\label{1.3}
\LiA(\fs_i) = -\mathbbm{1}_{\ft \neq \emptyset} \cdot  \LiA\left( \fh, q, \ft \right) - \mathbbm{1}_{\fh \neq \emptyset} \cdot D_1 \cdot \LiA\left( \fh^{+}, (q-1)*\ft \right) - D_1 \cdot \LiA\left( \fh, 1, (q-1)*\ft \right).
\end{equation}
After expanding the stuffle product, each tuple appearing on the right-hand side of \eqref{1.3} is $\succ (\fh,{}^{q+}\ft) = \fs_i$ and has weight $w$.
Therefore, \eqref{1.3} implies
\[
\LiA(\fs_i) \in \sum_{j=1}^{i-1} \bbF_p[D_1]_{\le 1} \LiA(\fs_j).
\]
So \eqref{1.2} follows by the induction hypothesis.
\end{proof}

\begin{corollary}\label{cor2.4}
For any $\fs_1,\fs_2 \in \cI$ with $\wt(\fs_1)+\wt(\fs_2)=w \ge 1$, we have
\[
\LiA(\fs_1)\LiA(\fs_2) \in \sum_{\ft \in \mathcal{B}_w^{\textup{T}}} \bbF_p[D_1]_{< 2^w} \LiA(\ft).
\]
\end{corollary}

\begin{proof}
Since $\LiA(\fs_1)\LiA(\fs_2)=\LiA(\fs_1*\fs_2)$ and $\fs_1*\fs_2$ is an $\bbF_p$-linear combination of tuples in $\cI_w$, we have
\[
\LiA(\fs_1)\LiA(\fs_2) \in \sum_{\fs \in \cI_w} \bbF_p\LiA(\fs).
\]
The corollary now follows from Lemma \ref{lem2.3}.
\end{proof}

\subsection{Deconcatenation coproduct}

In this subsection, we introduce the deconcatenation coproduct $\Delta$, prove its compatibility with the stuffle product, and compute its action on carry relations.

\begin{definition}[deconcatenation coproduct]
Define the $K$-linear map $\Delta$: $\sS \rightarrow \sS \otimes_K \sS$ by setting
\[
\Delta(\fs) \coloneq \sum_{\substack{\fh, \ft \in \cI \\ (\fh,\ft)=\fs}} \fh \otimes \ft
\]
for any $\fs \in \cI$.
For example, $\Delta(2,6,1) = \emptyset \otimes (2,6,1) + (2) \otimes (6,1) + (2,6) \otimes (1) + (2,6,1) \otimes \emptyset$.
\end{definition}

\begin{lemma}[compatibility of $\Delta$ with $*$]\label{lem_DeltaStuffle}
For any $\fs_1,\fs_2 \in \cI$, we have
\[
\Delta(\fs_1*\fs_2) = \sum_{\substack{\fh_1,\fh_2,\ft_1,\ft_2 \in \cI \\ (\fh_1,\ft_1)=\fs_1, (\fh_2,\ft_2)=\fs_2}} (\fh_1 * \fh_2) \otimes (\ft_1 * \ft_2).
\]
\end{lemma}

\begin{proof}
Hoffman \cite[Theorem 3.1]{Hoffman2000} proved the corresponding identity in characteristic $0$, which descends to our case, since all the coefficients descend from $\mathbb{Z}$ to $\bbF_p \subset K$.
\end{proof}

\begin{lemma}[\text{coproduct formula for carry relations}]\label{lem_DeltaCarry}
For any $\fh,\ft \in \cI$, we have
\begin{align*}
\Delta\big( \brho(\fh;\ft) \big) = &\sum_{\substack{\fh_1,\fh_2 \in \cI \\ (\fh_1,\fh_2)=\fh}} \fh_1 \otimes \brho(\fh_2;\ft) + \sum_{\substack{\ft_1,\ft_2 \in \cI \\ (\ft_1,\ft_2)=\ft}} \brho(\fh;\ft_1) \otimes \ft_2 \\
&+ D_1 \cdot \mathbbm{1}_{\fh \neq \emptyset} \cdot \sum_{\substack{\ft_1,\ft_2 \in \cI \\ (\ft_1,\ft_2)=\ft}}  (\fh^{+},\ft_1) \otimes \big((q-1)*\ft_2\big) + D_1 \cdot \sum_{\substack{\ft_1,\ft_2 \in \cI \\ (\ft_1,\ft_2)=\ft}}  (\fh,1,\ft_1) \otimes \big((q-1)*\ft_2\big).  
\end{align*}
\end{lemma}

\begin{proof}
This is a straightforward calculation.
To make the calculation more compact, we borrow the notation $\sfE$ and $\sfC$ from \cite[Definition 8.5]{HHLL2026+}. 
Define the $K$-linear operators $\sfE,\sfC: \sS \rightarrow \sS$ by setting, for each $\fs \in \cI$, 
\begin{align*}
\sfE\fs &\coloneq \mathbbm{1}_{\fs \neq \emptyset} \cdot \fs^{+} + (\fs,1), \\
\sfC\fs &\coloneq {}^{q+}\fs + \mathbbm{1}_{\fs \neq \emptyset} \cdot (q,\fs).
\end{align*}
Then, \eqref{def_rho} can be rewritten as
\[
\brho(\fh;\ft) = (\fh,\sfC\ft) + D_1 \cdot \left( \sfE\fh,(q-1)*\ft \right).
\]

Noting that 
\begin{align*}
\Delta(\sfC\ft) &= \emptyset \otimes \sfC\ft + \sum_{\substack{\ft_1,\ft_2 \in \cI \\ (\ft_1,\ft_2)=\ft}} \sfC\ft_1 \otimes \ft_2, \\
\Delta(\sfE\fh) &= \sfE\fh \otimes \emptyset + \sum_{\substack{\fh_1,\fh_2 \in \cI \\ (\fh_1,\fh_2)=\fh}} \fh_1 \otimes \sfE\fh_2, \\
\Delta\big( (q-1)*\ft \big) &= \sum_{\substack{\ft_1,\ft_2 \in \cI \\ (\ft_1,\ft_2)=\ft}} \Big( \big( (q-1)*\ft_1 \big) \otimes \ft_2 + \ft_1 \otimes \big( (q-1) * \ft_2 \big) \Big),
\end{align*}
we deduce 
\begin{align*}
\Delta\big(\brho(\fh;\ft)\big) =& \sum_{\substack{\fh_1,\fh_2 \in \cI \\ (\fh_1,\fh_2)=\fh}} \fh_1 \otimes (\fh_2,\sfC\ft)  + \sum_{\substack{\ft_1,\ft_2 \in \cI \\ (\ft_1,\ft_2)=\ft}} (\fh,\sfC\ft_1) \otimes \ft_2 \\
&+ D_1 \sum_{\substack{\fh_1,\fh_2 \in \cI \\ (\fh_1,\fh_2)=\fh}} \fh_1 \otimes \big(\sfE\fh_2,(q-1)*\ft\big) + D_1 \sum_{\substack{\ft_1,\ft_2 \in \cI \\ (\ft_1,\ft_2)=\ft}} \big( \sfE\fh, (q-1)*\ft_1 \big) \otimes \ft_2 \\
&+ D_1 \sum_{\substack{\ft_1,\ft_2 \in \cI \\ (\ft_1,\ft_2)=\ft}} (\sfE\fh,\ft_1) \otimes \big( (q-1)*\ft_2 \big) \\
=& \sum_{\substack{\fh_1,\fh_2 \in \cI \\ (\fh_1,\fh_2)=\fh}} \fh_1 \otimes \brho(\fh_2;\ft) + \sum_{\substack{\ft_1,\ft_2 \in \cI \\ (\ft_1,\ft_2)=\ft}} \brho(\fh;\ft_1) \otimes \ft_2 + D_1  \sum_{\substack{\ft_1,\ft_2 \in \cI \\ (\ft_1,\ft_2)=\ft}} (\sfE\fh,\ft_1) \otimes \big( (q-1)*\ft_2 \big),
\end{align*}
which is the desired identity.
\end{proof}

\section{Derivations}\label{sec_3}

In this section, we introduce a family of derivations on $\cZ$.
This is a key observation contributed by GPT-6 Astra in this paper.

Define $\cZ_{>0} \coloneq \bigoplus_{w=1}^{\infty} \cZ_w$. 
For any two $K$-linear subspaces $\mathcal{X}$, $\mathcal{Y}$ of $\cZ$, define
\[
\mathcal{X}\mathcal{Y} \coloneq \operatorname{Span}_{K} \left\{ xy \mid x \in \mathcal{X}, y \in \mathcal{Y}  \right\} \subset \cZ.
\]
A $K$-linear map $\partial: \cZ \rightarrow \cZ$ is called a $K$-derivation on $\cZ$ if it satisfies the Leibniz rule:
\[
\partial(xy) = \partial(x)y + x\partial(y), \quad x,y \in \cZ.
\] 
Since $1 \in K=\cZ_0 \subset \cZ$, every $K$-derivation $\partial$ satisfies $\partial(1)=0$ and hence $\partial(K)=\{0\}$.

\begin{lemma}\label{lem_deri}
Let $\ell$: $\cZ \rightarrow K$ be a $K$-linear functional such that 
\[
\ell(1) =0, \quad \ell\left( \cZ_{>0}\cZ_{>0} \right) = \{ 0 \}, \quad\text{and}\quad \ell(\LiA(q-1)) = 0.
\]
Then, there exists a $K$-derivation $\partial_{\ell}$ on $\cZ$ such that
\begin{equation}\label{2.1}
\partial_{\ell} \left( \LiA(\fs) \right) = \sum_{\substack{\fh, \ft \in \cI \\ (\fh,\ft)=\fs}}  \LiA(\fh) \cdot \ell\left(\LiA(\ft)\right) 
\end{equation}
for any $\fs \in \cI$, and $\partial_{\ell}\left( \LiA(q-1) \right) = 0$.
\end{lemma}

\begin{proof}
First, let
\[
\varepsilon: \mathcal{Z} \longrightarrow K
\]
be the projection onto the weight-zero part. 
Because $\mathcal{Z}=K \oplus \mathcal{Z}_{>0}$, every $x, y \in \mathcal{Z}$ can be written uniquely as
\[
x=\varepsilon(x)+x_{+}, \quad y=\varepsilon(y)+y_{+}, \quad x_{+}, y_{+} \in \mathcal{Z}_{>0} .
\]
Expanding their product gives
\[
x y=\varepsilon(x) \varepsilon(y)+\varepsilon(x) y_{+}+\varepsilon(y) x_{+}+x_{+} y_{+}.
\]
Now $\ell$ kills the first term because $\ell(1)=0$, and the last term because $\ell\left(\mathcal{Z}_{>0}\mathcal{Z}_{>0}\right)=\{0\}$. Consequently,
\begin{equation}\label{2.2}
\ell(xy) = \varepsilon(x)\ell(y) + \ell(x)\varepsilon(y), \quad x,y \in \cZ.
\end{equation}
In particular, taking $y=\LiA(q-1)$ in \eqref{2.2} and using the condition $\ell(\LiA(q-1))=0$, we obtain
\begin{equation}\label{2.3}
\ell\Big(\LiA(q-1) \cdot x \Big) =0, \quad x \in \cZ.
\end{equation}

Now define a $K$-linear map 
\[
\widetilde{\partial}_\ell: \mathscr{S} \longrightarrow \mathcal{Z}
\]
by setting, for each $\fs \in \cI$, 
\begin{equation}\label{2.4}
\widetilde{\partial}_\ell (\fs) \coloneq \left( \LiA \otimes (\ell \circ \LiA) \right)(\Delta(\fs)) =\sum_{\substack{\fh, \ft \in \cI \\ (\fh,\ft)=\fs}} \LiA(\fh) \cdot \ell\left(\LiA(\ft)\right).
\end{equation}
We claim that 
\begin{equation}\label{eqn_formal_Leib}
\widetilde{\partial}_\ell(\fs_1 * \fs_2) = \widetilde{\partial}_\ell(\fs_1) \LiA(\fs_2) + \LiA(\fs_1)\widetilde{\partial}_\ell(\fs_2), \quad \fs_1,\fs_2 \in \cI.
\end{equation}
In fact, applying the operator $\LiA \otimes (\ell \circ \LiA)$ to the identity in Lemma \ref{lem_DeltaStuffle}, and then using \eqref{2.2}, we have 
\begin{align*}
\widetilde{\partial}_\ell(\fs_1 * \fs_2) &= \sum_{\substack{\fh_1,\fh_2,\ft_1,\ft_2 \in \cI \\ (\fh_1,\ft_1)=\fs_1, (\fh_2,\ft_2)=\fs_2}} \LiA(\fh_1)\cdot \LiA(\fh_2) \cdot \ell\left(\LiA(\ft_1)\LiA(\ft_2)\right) \\
&= \sum_{\substack{\fh_1,\fh_2,\ft_1,\ft_2 \in \cI \\ (\fh_1,\ft_1)=\fs_1, (\fh_2,\ft_2)=\fs_2}} \LiA(\fh_1)\cdot \LiA(\fh_2) \cdot \Big( \varepsilon\left(\LiA(\ft_1)\right) \cdot \ell\left(\LiA(\ft_2)\right) + \ell\left(\LiA(\ft_1)\right) \cdot \varepsilon\left(\LiA(\ft_2)\right) \Big).
\end{align*}
Since $\varepsilon\left(\LiA(\ft)\right) = \mathbbm{1}_{\ft=\emptyset}$, we deduce that
\begin{align*}
\widetilde{\partial}_\ell(\fs_1 * \fs_2) &= \LiA(\fs_1) \sum_{\substack{\fh_2,\ft_2 \in \cI \\ (\fh_2,\ft_2)=\fs_2}} \LiA(\fh_2) \cdot \ell\left(\LiA(\ft_2)\right) + \LiA(\fs_2) \sum_{\substack{\fh_1,\ft_1 \in \cI \\ (\fh_1,\ft_1)=\fs_1}} \LiA(\fh_1) \cdot \ell\left(\LiA(\ft_1)\right) \\
&= \LiA(\fs_1)\widetilde{\partial}_\ell(\fs_2) + \widetilde{\partial}_\ell(\fs_1) \LiA(\fs_2),
\end{align*}
which proves \eqref{eqn_formal_Leib}. 
By bilinearity, \eqref{eqn_formal_Leib} extends to 
\begin{equation}\label{2.6}
\widetilde{\partial}_\ell(\bS_1 * \bS_2) = \widetilde{\partial}_\ell(\bS_1) \LiA(\bS_2) + \LiA(\bS_1)\widetilde{\partial}_\ell(\bS_2), \quad \bS_1,\bS_2 \in \sS.
\end{equation}

We now prove that $\widetilde{\partial}_\ell$ kills any carry relation.
Indeed, applying the operator $\LiA \otimes (\ell \circ \LiA)$ to the identity in Lemma \ref{lem_DeltaCarry} gives, for all $\fh,\ft \in \cI$,
\begin{align*}
\widetilde{\partial}_\ell\big( \brho(\fh;\ft) \big) = &\sum_{\substack{\fh_1,\fh_2 \in \cI \\ (\fh_1,\fh_2)=\fh}} \LiA(\fh_1) \cdot \ell\Big( \LiA\big(\brho(\fh_2;\ft)\big) \Big) + \sum_{\substack{\ft_1,\ft_2 \in \cI \\ (\ft_1,\ft_2)=\ft}} \LiA\big(\brho(\fh;\ft_1)\big) \cdot \ell\Big( \LiA(\ft_2) \Big) \\
&+ D_1 \cdot \mathbbm{1}_{\fh \neq \emptyset} \cdot \sum_{\substack{\ft_1,\ft_2 \in \cI \\ (\ft_1,\ft_2)=\ft}}  \LiA(\fh^{+},\ft_1) \cdot \ell\Big(\LiA(q-1) \cdot \LiA(\ft_2) \Big) \\
&+ D_1 \cdot \sum_{\substack{\ft_1,\ft_2 \in \cI \\ (\ft_1,\ft_2)=\ft}}  \LiA(\fh,1,\ft_1) \cdot \ell\Big( \LiA(q-1) \cdot \LiA(\ft_2) \Big). 
\end{align*}  
By Lemma \ref{lem_gIKN} and Equation \eqref{2.3}, each summand on the right-hand side above is zero. 
Thus, 
\[
\widetilde{\partial}_\ell\big( \brho(\fh;\ft) \big) = 0, \quad \fh,\ft \in \cI.
\]
Therefore, by Theorem \ref{lem_kerLi}, 
\[
\ker\left( \widetilde{\partial}_\ell: \sS \rightarrow \cZ \right) \supset \ker \left( \LiA: \sS \rightarrow \cZ \right).
\] 
Consequently, $\widetilde{\partial}_\ell$ induces a $K$-linear map $\partial_{\ell}: \cZ \rightarrow \cZ$ satisfying $\partial_{\ell}(\LiA(\fs)) = \widetilde{\partial}_\ell(\fs)$ for every $\fs \in \cI$.
Equation \eqref{2.4} descends to Equation \eqref{2.1}.
Equation \eqref{2.6} descends to the Leibniz rule for $\partial_{\ell}$.
By \eqref{2.4} and $\ell(1)=\ell(\LiA(q-1))=0$, we have $\widetilde{\partial}_\ell((q-1))=0$, which descends to $\partial_{\ell}(\LiA(q-1))=0$. 
Thus, $\partial_{\ell}$ satisfies all the requirements.
\end{proof}

\section{The polynomiality of $\cZ$}\label{sec_4}

In this section, we prove the polynomiality of $\cZ$.
The key tool is Lemma \ref{lem_deri}.
We first select the generator candidates by lifting a homogeneous basis of a certain quotient space.
Then, Lemma \ref{lem_deri} provides many $K$-derivations on $\cZ$, so that we can apply them to a hypothetical algebraic relation among the selected generators to derive a contradiction.

An element $z \in \cZ = \bigoplus_{w=0}^{\infty} \cZ_w$ is called a homogeneous element if $z \in \cZ_w$ for some $w$.
For any $z \in \cZ_w \setminus \{0\}$, define $\wt(z)=w$.
This notation is compatible with the weight notation on $\cI$: $\wt(\LiA(\fs))=\wt(\fs)$ for any $\fs \in \cI$.

\begin{theorem}\label{thm4.1}
Let 
\[
z_0=\LiA(q-1) \quad\text{and}\quad V=\cZ_{>0}/\left(\cZ_{>0}\cZ_{>0}+ Kz_0\right).
\]
Choose any set of homogeneous elements $\{y_i \}_{i \in I} \subset \cZ_{>0}$ such that their classes $\overline{y_i} = y_i + \left(\cZ_{>0}\cZ_{>0}+ Kz_0\right)$ form a $K$-linear basis of $V$. 
Then, the elements $z_0$ and $y_i$ $(i \in I)$ freely generate the algebra $\cZ$. 
That is, $\cZ = K[z_0, \{y_i\}_{i \in I}]$ and $\{z_0\} \cup \{y_i\}_{i \in I}$ is an algebraic independent set over $K$.
In particular, $\cZ$ is a graded polynomial algebra.
\end{theorem}

\begin{proof}
First, we prove the generation property 
\begin{equation}\label{4.0}
K[z_0, \{y_i\}_{i \in I}] = \cZ.
\end{equation}
We show by induction on $w$ that $\cZ_{w} \subset K[z_0, \{y_i\}_{i \in I}]$. 
The base case $w=0$ follows from $\cZ_0 = K$.
Let $x \in \cZ_{w}$ with $w>0$.
Since $\{\overline{y}_i\}_{i \in I}$ is a $K$-linear basis of $V = \cZ_{>0}/\left(\cZ_{>0}\cZ_{>0}+ Kz_0\right)$, and $y_i$ are homogeneous elements, there exist coefficients $c_i, c \in K$, where $c=0$ unless $w=q-1$, such that
\[
x - \sum_{\substack{i \in I \\ \wt(y_i)=w}} c_i y_i - cz_0 \in \left( \cZ_{>0} \cZ_{>0} \right) \cap \cZ_w = \sum_{n=1}^{w-1} \cZ_n\cZ_{w-n}.
\]
By the induction hypothesis, $\cZ_{n} \subset K[z_0, \{y_i\}_{i \in I}]$ for any $n<w$. 
Therefore, $x-\sum c_iy_i - cz_0 \in K[z_0, \{y_i\}_{i \in I}]$ and thus $x \in K[z_0, \{y_i\}_{i \in I}]$. 
This completes the induction step and hence proves $\cZ = K[z_0, \{y_i\}_{i \in I}]$.

Next, we construct derivations that distinguish the $y_i$. For each $i \in I$, define the $K$-linear map $\ell_i: \cZ \rightarrow K$ as follows:
it vanishes on $\cZ_0=K$, and on $\cZ_{>0}$ it extracts the coefficient of $\overline{y}_i$ in the basis expansion in $V$. 
Thus, for any $i,j \in I$,
\[
\ell_i(y_j) = \delta_{i,j}, \quad \ell_i\left(\cZ_{>0} \cZ_{>0}\right) =\{0\}, \quad \ell_i(z_0) =0, \quad\text{and}\quad \ell_i(1)=0,
\]
where $\delta_{i,j} \in \{0,1\} \subset K$ is the Kronecker delta.
Since $y_j$ are homogeneous elements, $\ell_{i}$ is supported on $\cZ_{\wt(y_i)}$.
Now, Lemma \ref{lem_deri} gives a $K$-derivation
\[
\partial_i = \partial_{\ell_i}: \cZ \rightarrow \cZ
\]
for each $i \in I$.
We claim that
\begin{equation}\label{4.1}
\partial_i(z_0)=0, \quad \partial_i\left( \cZ_w \right) = \{0\} \text{ if } w<\wt(y_i), \quad\text{and } \partial_i(x) = \ell_i(x) \text{ if } x \in \cZ_{\wt(y_i)}.
\end{equation}
Indeed, $\partial_i(z_0)=0$ follows from the conclusion of Lemma \ref{lem_deri}.
Since $\ell_i$ is supported on $\cZ_{\wt(y_i)}$, in
\[
\partial_{i} \big( \LiA(\fs) \big) = \sum_{\substack{\fh, \ft \in \cI \\ (\fh,\ft)=\fs}}  \LiA(\fh) \cdot \ell_i\big(\LiA(\ft)\big), \quad \fs \in \cI,
\]
only terms with $\wt(\ft) = \wt(y_i)$ contribute to the sum. 
This proves the latter two assertions in \eqref{4.1}. 
Consequently, 
\begin{equation}\label{4.2}
\partial_i(z_0) =0, \quad\text{and } \partial_i\left(y_j\right)= \begin{cases}0, &\text{if } \wt(y_j)<\wt(y_i), \\ \delta_{i,j}, &\text{if } \wt(y_j)=\wt(y_i) .\end{cases} 
\end{equation}

Finally, we prove the $K$-algebraic independence of $\{z_0\} \cup \{y_i\}_{i \in I}$. 
Suppose that a polynomial relation exists. 
It involves only finitely many of the $y_i$; after relabelling, fix these as
\[
y_1, \ldots, y_m, \quad \wt(y_1) \le \cdots \le \wt(y_m).
\]
Assign the weighted degree
\[
\operatorname{deg} X_0= q-1, \quad \operatorname{deg} X_i=\wt(y_i), \quad i=1,\ldots,m,
\]
to the polynomial ring
\[
K\left[X_0, X_1, \ldots, X_m\right].
\]
Evaluation at $\left(z_0, y_1, \ldots, y_m\right)$ respects weights. 
Therefore, by Theorem \ref{thm_Chang}, every homogeneous component of a polynomial relation is itself a relation. 
Choose a nontrivial homogeneous polynomial relation
\begin{equation}\label{4.3}
F\left(z_0, y_1, \ldots, y_m\right)=0 
\end{equation}
of the smallest possible weighted degree $N = \deg F$ in the polynomial ring $K\left[X_0, X_1, \ldots, X_m\right]$. 
Necessarily $N>0$, since a nonzero scalar cannot evaluate to zero.

Apply $\partial_i$ to \eqref{4.3}. 
Since $\partial_i$ kills coefficients in $K$ and kills $z_0$, the chain rule gives
\begin{equation}\label{4.4}
\sum_{j=1}^m \frac{\partial F}{\partial X_j}\left(z_0, y_1, \ldots, y_m\right) \partial_i\left(y_j\right) = 0, \quad i=1,\ldots,m. 
\end{equation}
By \eqref{4.2}, the matrix 
\[
\Big(\partial_i\left(y_j\right)\Big)_{1 \le i, j \le m}
\]
is upper triangular, with every diagonal entry equal to $1$.
Therefore, $\det \big(\partial_i\left(y_j\right)\big)_{1 \le i, j \le m} = 1$, so \eqref{4.4} implies
\begin{equation}\label{4.5}
\frac{\partial F}{\partial X_j}\left(z_0, y_1, \ldots, y_m\right)=0, \quad j=1,\ldots,m.
\end{equation}
We must also handle differentiation with respect to $X_0$, corresponding to $z_0$. 
Since $F$ is homogeneous, the weighted Euler identity is
\begin{equation}\label{4.6}
N F = (q-1) X_0 \frac{\partial F}{\partial X_0}+\sum_{j=1}^m \wt(y_j) X_j \frac{\partial F}{\partial X_j}.
\end{equation}
This follows term by term for each monomial in $F$ and is valid in characteristic $p$.

Evaluating \eqref{4.6} at $(z_0,y_1,\ldots,y_m)$, and using \eqref{4.3} and \eqref{4.5}, we obtain
\[
(q-1) z_0 \frac{\partial F}{\partial X_0}\left(z_0, y_1, \ldots, y_m\right)=0.
\]
Since $q-1 \neq 0 \in K_{\infty}$ and $z_0 = \LiA(q-1)=\zeta_A(q-1) \neq 0 \in K_{\infty}$, 
\begin{equation}\label{4.7}
\frac{\partial F}{\partial X_0}\left(z_0, y_1, \ldots, y_m\right)=0.
\end{equation}
Every nonzero formal partial derivative of $F$ is homogeneous of weight strictly smaller than $N$. 
Equations \eqref{4.5} and \eqref{4.7}, together with the minimal choice of $N$, therefore imply
\begin{equation}\label{4.8}
\frac{\partial F}{\partial X_j}=0 \text { as a polynomial, for every } j=0,1,\ldots,m. 
\end{equation}
Equation \eqref{4.8} tells us precisely that
\begin{equation}\label{4.9}
F \in K\left[X_0^p, X_1^p, \ldots, X_m^p\right]. 
\end{equation}
Indeed, any monomial in the homogeneous polynomial $F$ with an exponent not divisible by $p$ would give a nonzero contribution to the corresponding partial derivative that cannot be cancelled.

We cannot immediately write $F=G^p$ for some polynomial $G$, because $K=\mathbb{F}_q(\theta)$ is imperfect. 
Instead, use
\begin{equation}\label{4.10}
K=\bigoplus_{r=0}^{p-1} \theta^r K^p.
\end{equation}
Here $K^p=\mathbb{F}_q\left(\theta^p\right)$, since $\mathbb{F}_q$ is perfect. 
To prove the existence of the decomposition, write a rational function as
\[
\frac{A(\theta)}{B(\theta)}=\frac{A(\theta) B(\theta)^{p-1}}{B(\theta)^p}
\]
and group the numerator's powers of $\theta$ according to their exponents modulo $p$. 
Uniqueness follows by clearing denominators and comparing these residue classes.
By \eqref{4.9} and \eqref{4.10}, we can write
\[
F=\sum_{\boldsymbol{\beta}} c_{\boldsymbol{\beta}} X_0^{p \beta_0} \cdots X_m^{p \beta_m}
\]
and decompose each coefficient $c_{\boldsymbol{\beta}} \in K$ uniquely as
\[
c_{\boldsymbol{\beta}} = \sum_{r=0}^{p-1} \theta^r b_{r, \boldsymbol{\beta}}^p, \quad b_{r, \boldsymbol{\beta}} \in K.
\]
Setting
\[
G_r=\sum_{\boldsymbol{\beta}} b_{r, \boldsymbol{\beta}} X_0^{\beta_0} \cdots X_m^{\beta_m} \in K[X_0,\ldots,X_m],
\]
we obtain
\begin{equation}\label{4.11}
F=\sum_{r=0}^{p-1} \theta^r G_r^p. 
\end{equation}
Since $F$ is nonzero, at least one $G_r$ is nonzero. 
Moreover, each nonzero $G_r$ is homogeneous of weighted degree $N / p <N$.
The corresponding decomposition of $K_{\infty}=\mathbb{F}_q\left(\!\left(\theta^{-1}\right)\!\right)$ is
\begin{equation}\label{4.12}
K_{\infty}=\bigoplus_{r=0}^{p-1} \theta^r K_{\infty}^p, \quad K_{\infty}^p=\mathbb{F}_q\left(\!\left(\theta^{-p}\right)\!\right).
\end{equation}
Equation \eqref{4.12} follows by grouping Laurent exponents modulo $p$. 
The Laurent exponents occurring in $\theta^r K_{\infty}^p$ are congruent to $r$ modulo $p$. 
Terms belonging to different residue classes cannot cancel.

Evaluating \eqref{4.11} at $(z_0,y_1,\ldots,y_m)$, we get
\[
\sum_{r=0}^{p-1} \theta^r G_r\left(z_0, y_1, \ldots, y_m\right)^p =0.
\]
Since the sum in \eqref{4.12} is direct, we have
\[
G_r\left(z_0, y_1, \ldots, y_m\right) =0, \quad r=0,1,\ldots,p-1.
\]
A nonzero $G_r$ therefore gives an algebraic relation of weight $N / p<N$, contradicting the minimality of $N$.
This proves the algebraic independence of $\{z_0\} \cup \{y_i\}_{i \in I}$ over $K$. 
Together with \eqref{4.0}, we obtain the $K$-algebra isomorphism 
\[
K\left[X_0, \left\{X_i\right\}_{i \in I} \right] \simeq \mathcal{Z}, \quad X_0 \longmapsto z_0, \quad X_i \longmapsto y_i.
\]
Therefore, $\cZ$ is a graded polynomial algebra.
\end{proof}

The preceding proof uses the distinguished value $\LiA(q-1)$ to prove the polynomiality of $\cZ$. 
This is mainly because in Lemma \ref{lem_deri} we need the condition $\ell(\LiA(q-1)) = 0$ to ensure that $\widetilde{\partial}_{\ell}$ kills all carry relations and thereby descends to a $K$-derivation $\partial_{\ell}$ on $\cZ$. 
However, once Theorem \ref{thm4.1} is proved, we can easily remove the particular role of $\LiA(q-1)$ and deduce the following structural theorem of $\cZ$.
Recall the notation $\cZ_{\le w}^{\textup{alg}} = K[\zeta_A(\fs) : \wt(\fs) \le w]$.

\begin{theorem}\label{thm4.2}
For any weight $n \ge 1$, let 
\[
Q_n = \cZ_{n} \Big/ \sum_{r=1}^{n-1} \cZ_{r}\cZ_{n-r} \quad\text{and}\quad b_n = \dim_{K} Q_n.
\]
Choose any elements $t_{n,1},\ldots,t_{n,b_n} \in \cZ_{n}$ such that their classes $\overline{t_{n,j}}=t_{n,j} + \sum_{r=1}^{n-1} \cZ_{r}\cZ_{n-r}$ $(j=1,\ldots,b_n)$ form a basis of $Q_n$.
Then, the elements $t_{n,j}$ $(n \in \bbZ_+, 1 \le j \le b_n)$ freely generate the $K$-algebra $\cZ$. 
Moreover, for any $w \ge 1$, the elements $t_{n,j}$ $(1 \le n \le w, 1 \le j \le b_n)$ freely generate the $K$-algebra $\cZ_{\le w}^{\textup{alg}}$.
\end{theorem}

\begin{proof}
Retain the notation $z_0$ and $\{y_i\}_{i \in I}$ as in the statement of Theorem \ref{thm4.1}.
We have proved that
\[
\cZ = K[z_0, \{y_i\}_{i \in I}] \quad\text{freely}.
\]
In this polynomial algebra, $\cZ_{>0}$ is spanned by all nonconstant monomials, and $\cZ_{>0}\cZ_{>0}$ is spanned by monomials containing at least two generators, counted with multiplicity. 
It follows that
\[
\left\{ \overline{z_0} \right\} \cup \left\{ \overline{y_i} \right\}_{i \in I} \text{ is a $K$-basis of } \cZ_{>0}/\cZ_{>0}\cZ_{>0},
\]
where the bars denote classes in $\cZ_{>0}/\cZ_{>0}\cZ_{>0}$.
Therefore, for each weight $n \ge 1$, if we rename the weight-$n$ elements in $\{ z_0 \} \cup \{y_i\}_{i \in I}$ as $u_{n,1},\ldots,u_{n,b_n^\prime}$, then their classes $\overline{u_{n,1}},\ldots,\overline{u_{n,b_n^{\prime}}}$ in $Q_n$ form a $K$-basis of $Q_n$. 
Thus, we have $b_{n}^{\prime} = b_n$.

Let $\mathbf{u}_n$ and $\mathbf{t}_n$ denote the column vectors
\[
\mathbf{u}_n=\left(u_{n, 1}, \ldots, u_{n, b_n}\right)^{\top}, \quad \mathbf{t}_n=\left(t_{n, 1}, \ldots, t_{n, b_n}\right)^{\top} .
\]
Since $\{ \overline{u_{n,1}},\ldots,\overline{u_{n,b_n}} \}$ and $\{ \overline{t_{n,1}},\ldots,\overline{t_{n,b_n}} \}$ are two bases of $Q_n$, there is an invertible matrix
\[
\mathbf{A}_n \in \operatorname{GL}_{b_n}(K)
\]
such that
\[
\mathbf{t}_n-\mathbf{A}_n \mathbf{u}_n \in \left( \cZ_{>0}\cZ_{>0} \right)^{b_n}.
\]
Recall that $\cZ_{>0}\cZ_{>0}$ is spanned by monomials containing at least two generators $u_{w,j}$, counted with multiplicity. 
Each entry of $\mathbf{t}_n-\mathbf{A}_n \mathbf{u}_n$ belongs to $\cZ_n$, so the corresponding monomials use only generators $u_{w,j}$ with $w < n$.
Therefore,
\begin{equation}\label{4.13}
\mathbf{t}_n=\mathbf{A}_n \mathbf{u}_n+\mathbf{P}_n\left(\mathbf{u}_{<n}\right),
\end{equation}
where $\mathbf{P}_n\left(\mathbf{u}_{<n}\right)$ is a column vector of polynomials in generators $u_{w,j}$ of weights below $n$ with coefficients in $K$.
Since $\mathbf{A}_n$ is invertible, we can solve for $\mathbf{u}_n$:
\begin{equation}\label{4.14}
\mathbf{u}_n=\mathbf{A}_n^{-1}\Big(\mathbf{t}_n-\mathbf{P}_n\left(\mathbf{u}_{<n}\right)\Big). 
\end{equation}
Starting from the lowest weights, equation \eqref{4.14} recursively expresses every $u_{n,j}$ as a polynomial in $t_{w,j}$ ($1 \le w \le n$, $1 \le j \le b_w$). 
This proves that $t_{n,j}$ $(n \in \bbZ_+, 1 \le j \le b_n)$ generate the algebra $\cZ$. 

To prove algebraic independence, perform the same substitutions with independent formal variables $T_{n,j}$, $X_{n,j}$.
Let
\[
\Phi: K\left[\mathbf{T}_n: n \ge 1\right] \rightarrow K\left[\mathbf{X}_n: n \ge 1\right]
\]
be the $K$-algebra homomorphism defined by
\[
\Phi\left(\mathbf{T}_n\right)=\mathbf{A}_n \mathbf{X}_n+\mathbf{P}_n\left(\mathbf{X}_{<n}\right).
\]
Define a $K$-algebra homomorphism in the opposite direction $\Psi: K\left[\mathbf{X}_n: n \ge 1\right] \rightarrow K\left[\mathbf{T}_n: n \ge 1\right]$ recursively by
\begin{equation}\label{4.16}
\Psi\left(\mathbf{X}_n\right)=\mathbf{A}_n^{-1}\Big(\mathbf{T}_n-\mathbf{P}_n\big(\Psi\left(\mathbf{X}_{<n}\right)\big)\Big).
\end{equation}
This is well-defined: by induction on $n$ the right-hand side of \eqref{4.16} produces a column vector with entries in $K[\mathbf{T}_1,\ldots,\mathbf{T}_{n}]$.
Equation \eqref{4.16} immediately gives
\[
\Psi \Phi\left(\mathbf{T}_n\right)=\mathbf{T}_n.
\]
Induction on $n$ yields
\[
\Phi \Psi\left(\mathbf{X}_n\right)=\mathbf{X}_n.
\]
Thus $\Phi$ and $\Psi$ are inverse isomorphisms of polynomial rings.
Let $\iota$ denote the established isomorphism $K\left[\mathbf{X}_n : n \ge 1\right] \xrightarrow{\sim} \mathcal{Z}$, $X_{n,j} \mapsto u_{n,j}$. 
The map $\iota \circ \Phi$ sends each formal variable $T_{n,j}$ to the corresponding $t_{n,j}$. 
Hence these elements $t_{n,j}$ freely generate $\mathcal{Z}$.

Finally, since $t_{n,j} \in \cZ_{n}$, we have 
\[
K\left[t_{n,j}: 1 \le n \le w, 1 \le j \le b_n \right] \subset \cZ_{\le w}^{\textup{alg}}.
\]
Since $\overline{t_{m,1}}, \ldots, \overline{t_{m,b_m}}$ form a $K$-basis of $Q_m$, for every $x \in \cZ_{m}$, there exist $c_1,\ldots,c_{b_m} \in K$ such that
\[
x - \sum_{j=1}^{b_m} c_{j}t_{m,j} \in \sum_{r=1}^{m-1} \cZ_{r}\cZ_{m-r}.
\] 
Induction on $m$ gives $\cZ_{m} \subset K\left[t_{n,j}: 1 \le n \le w, 1 \le j \le b_n \right]$ for any $m=0,1,\ldots,w$.
Therefore,
\[
\cZ_{\le w}^{\textup{alg}} = K\left[t_{n,j}: 1 \le n \le w, 1 \le j \le b_n \right] \quad\text{freely},
\]
as asserted.
\end{proof}

\begin{lemma}\label{lem4.3}
The integer $a_n$ defined by \eqref{def_an} and the integer $b_n$ defined in the statement of Theorem \ref{thm4.2} are the same: $a_n=b_n$ for any $n \ge 1$. 
\end{lemma}

\begin{proof}
Retain the notation $t_{n,j}$ as in Theorem \ref{thm4.2}. 
We have proved that 
\[
\cZ = K[t_{n,j}: n \in \bbZ_+, 1 \le j \le b_n] \quad\text{freely}.
\] 
Therefore, the weight-$w$ monomials in these generators $t_{n,j}$ form a $K$-linear basis of $\cZ_w$ for each $w$. 
Consequently,
\[
\prod_{n=1}^{\infty} \frac{1}{(1-X^n)^{b_n}} = \sum_{w=0}^{\infty} \left(\dim_K \cZ_w\right) X^w \quad \text{as formal power series in } \bbZ\llbracket X \rrbracket.
\]
By Theorem \ref{thm_CCMIKLNP}, we obtain
\[
\prod_{n=1}^{\infty} \frac{1}{(1-X^n)^{b_n}} = \frac{1-X^{q}}{1-X-X^2-\cdots-X^q}.
\]
This identity uniquely determines the integers $b_n$ and implies $a_n=b_n$.
\end{proof}

\section{An explicit polynomial basis}\label{sec_5}

In this section, we use an elementary determinant argument to prove that the explicitly constructed elements $z_{n,j}$ in \eqref{def_znj} meet the requirements of Theorem \ref{thm4.2}.
This leads to the following proof of our main theorem.

\begin{proof}[Proof of Theorem \ref{thm_main}]
For any weight $n \ge 1$, by Theorem \ref{thm_CCMIKLNP} and Lemma \ref{lem4.3}, 
\[
\dim_{K} \cZ_n = d_n, \quad \dim_K \left( \cZ_n \Big/ \sum_{r=1}^{n-1} \cZ_r\cZ_{n-r} \right) = a_n.
\]
Therefore, $\dim_{K} \sum_{r=1}^{n-1} \cZ_r\cZ_{n-r} = d_n - a_n$.
By Theorem \ref{thm_FpSpan}, 
\begin{equation}\label{5.1}
\left\{ \LiA(\fs_1)\LiA(\fs_2) : \fs_1,\fs_2 \in \cI_{>0}, \wt(\fs_1)+\wt(\fs_2)=n \right\}
\end{equation}
is a spanning set of the $K$-linear space $\sum_{r=1}^{n-1} \cZ_r\cZ_{n-r}$.
Therefore, there is a subset of \eqref{5.1} of cardinality $d_n-a_n$, denoted by $\{ \gamma_{n,a_n+1}, \gamma_{n,a_n+2}, \ldots, \gamma_{n,d_n} \}$, that forms a $K$-basis of $\sum_{r=1}^{n-1} \cZ_r\cZ_{n-r}$.

By Corollary \ref{cor2.4}, there exist $c_{i,j} \in \bbF_p[D_1]_{<2^{n}}$ (for $1 \le i \le d_n$ and $a_n < j \le d_n$) such that 
\[
\gamma_{n,j} = \sum_{i=1}^{d_n} c_{i,j}\LiA(\ft_{n,i}),
\]
where
\[
\ft_{n,1} \prec \ft_{n,2} \prec \cdots \prec \ft_{n,d_n} 
\]
are the Todd--Thakur tuples of weight $n$ listed in ascending order with respect to the (depth, lex)-order.
For $1 \le i \le d_n$ and $1 \le j \le a_n$, define $c_{i,j}=D_1^{i\cdot4^{nj}}$. 
Then, by \eqref{def_znj} and Remark \ref{rmk1.3}, 
\[
z_{n,j} = \sum_{i=1}^{d_n} c_{i,j}\LiA(\ft_{n,i}).
\] 

Now, consider the matrix $C=(c_{i,j})_{1 \le i,j \le d_n} \in \operatorname{Mat}_{d_n}(\bbF_p[D_1])$.
Expanding the determinant $\det C$ along the first $a_n$ columns gives
\[
\det C = \sum_{\substack{1 \le i_1,\ldots,i_{a} \le d \\ i_1,\ldots,i_{a} \text{ distinct} } } \epsilon_{\boldsymbol{i}} \Delta_{\boldsymbol{i}} D_1^{i_1 \cdot 4^n + i_2 \cdot 4^{2n} + \cdots + i_{a} \cdot 4^{an}},
\]
where $a=a_n$, $d=d_n$, $\epsilon_{\boldsymbol{i}} \in \{ \pm 1\}$, and 
\[
\Delta_{\boldsymbol{i}} = \det\Big( c_{i,j} \Big)_{\substack{i \in \{1,\ldots,d\} \setminus \{i_1,\ldots,i_a\} \\ a < j \le d }}.
\]
Since $c_{i,j} \in \bbF_p[D_1]_{<2^n}$ for $a_n < j \le d_n$, and $2^n \cdot (d_n-a_n) < 4^{n}$, we have
\[
\Delta_{\boldsymbol{i}} \in \bbF_p[D_1]_{<4^n}.
\]
Since $i_1,\ldots, i_{a_n} \le d_n < 4^n$, the integers
\[
N_{\boldsymbol{i}} \coloneq i_1 + i_2 \cdot 4^{n} + \cdots + i_{a_n} \cdot 4^{(a_n-1)n}
\]
are distinct for distinct tuples $\boldsymbol{i}=(i_1,\ldots,i_{a_n})$. 
Consequently, the powers of $D_1$ occurring in the $\boldsymbol{i}$-th summand lie in the interval 
\[
\Big[4^nN_{\boldsymbol{i}}, 4^n\left( N_{\boldsymbol{i}} + 1\right) \Big).
\]
These intervals are pairwise disjoint. 
Thus cancellation between different summands is impossible.
At least one $\Delta_{\boldsymbol{i}}$ is nonzero, because $\gamma_{n,a_n+1}, \gamma_{n,a_n+2}, \ldots, \gamma_{n,d_n}$ are linearly independent over $K$ by definition.
Therefore, 
\[
\det C \neq 0.
\]
Hence $z_{n,1},\ldots,z_{n,a_n}$, $\gamma_{n,a_n+1}, \gamma_{n,a_n+2}, \ldots, \gamma_{n,d_n}$ form a $K$-basis of $\cZ_n$.
It follows that $\overline{z_{n,1}}, \ldots, \overline{z_{n,a_n}}$ form a $K$-basis of $\cZ_n/\sum_{r=1}^{n-1} \cZ_r\cZ_{n-r}$.
Thus, Theorem \ref{thm4.2} implies Theorem \ref{thm_main}.
\end{proof}

\vspace*{3mm}
\begin{flushright}
\begin{minipage}{158mm}\sc\footnotesize
School of Mathematical Sciences, Xiamen University, Fujian, China \\
{\it E-mail addresses}: \href{mailto:lilaimath@gmail.com}{{\tt lilaimath@gmail.com}}, \href{mailto:lilai@xmu.edu.cn}{{\tt lilai@xmu.edu.cn}} 
\vspace*{3mm}
\end{minipage}
\end{flushright}

\end{document}